\documentclass[11pt,reqno]{amsart}

\usepackage[T1]{fontenc}
\usepackage{lmodern}
\usepackage{microtype}
\usepackage{amsmath,amssymb,mathtools,mathrsfs}
\usepackage{enumitem}
\usepackage{xcolor}
\usepackage[colorlinks=true,linkcolor=blue,citecolor=blue,urlcolor=blue]{hyperref}

\newtheorem{theorem}{Theorem}[section]
\newtheorem{corollary}[theorem]{Corollary}
\newtheorem{lemma}[theorem]{Lemma}

\numberwithin{equation}{section}

\newcommand{\Czero}{\mathbb C_0}
\newcommand{\D}{\mathbb D}
\newcommand{\cD}{\mathcal D}
\newcommand{\BMAD}{BMOA_{\mathcal D}}
\newcommand{\VMAD}{VMOA_{\mathcal D}}
\newcommand{\LMAD}{LMOA_{\mathcal D}}
\newcommand{\dd}{\,d}
\newcommand{\Real}{\operatorname{Re}}

\title[Volterra operators on the BMOA space of Dirichlet series]
{Volterra operators on the BMOA space of Dirichlet series}

\author[C. Tong]{Cezhong Tong}
\address{Cezhong Tong, Institute of Mathematics, Hebei University of Technology,
Tianjin 300401, China}
\email{ctong@hebut.edu.cn}
\email{cezhongtong@hotmail.com}

\author[M. Wang]{Maofa Wang}
\address{Maofa Wang, School of Mathematics and Statistics, Wuhan University,
Wuhan 430072, China}
\email{mfwang.math@whu.edu.cn}

\author[Z. Yuan]{Zixing Yuan}
\address{Zixing Yuan, Institute of Mathematics, Hebei University of Technology,
Tianjin 300401, China}
\email{zxyuan.math@whu.edu.cn}
\date{\today}


\begin{document}

\begin{abstract}
In this paper, we completely characterize the boundedness and compactness of
Volterra operators on \(BMOA_{\mathcal D}\), the BMOA space of Dirichlet series, and estimate their essential and weak essential norms. We further
prove that compactness, weak compactness, and complete continuity are
equivalent for these operators.
In particular, every Dirichlet polynomial symbol induces a compact
Volterra operator on \(BMOA_{\mathcal D}\). We also
prove that every such symbol induces a compact Volterra operator on the
Bergman space of Dirichlet series \(\mathcal A_\mu^1\) for every admissible
measure \(\mu\), extending the corresponding compactness theorem,
previously established for \(1<p<\infty\), to the endpoint \(p=1\).

\end{abstract}

\maketitle
\section{Introduction}

Let \(\cD\) denote the class of ordinary Dirichlet series
\(f(s)=\sum_{n\geq1}a_n n^{-s}\) that converge in some right half-plane, and set
\(\Czero=\{s:\Real s>0\}\).  The space \(BMOA(\Czero)\) consists of the
holomorphic functions \(f\) in \(\Czero\) such that
\begin{align}\label{BMO}
 \sup_{\sigma>0}\int_{\mathbb R}
 \frac{|f(\sigma+it)|^2}{1+\sigma^2+t^2}\,\dd t<\infty
\end{align}
and
\[
 \|f\|_{\operatorname{BMO}(\Czero)}
 :=\sup_{I\subset\mathbb R}\frac1{|I|}\int_I
 \left|f(it)-\frac1{|I|}\int_I f(i\tau)\dd\tau\right|\dd t<\infty,
\]
where the supremum is taken over all intervals $I\subset\mathbb R$ and
$|I|$ denotes the length of $I$. Note that \eqref{BMO} implies
$F(s):=f(s)/(1+s)$ belongs to $H^2(\mathbb{C}_0)$,
and hence the boundary limit
$f(it)=\lim_{\sigma \rightarrow 0^+}f(\sigma+it)$
exists for almost all real $t$ by Fatou's theorem applied to $F(s)$.

We define $\BMAD:=BMOA(\Czero)\cap\cD$ and equip it with the norm \[ \|f\|_{\BMAD} :=|f(+\infty)| +\|f\|_{\operatorname{BMO}(\Czero)},\] where \(f(+\infty)=a_1\). We define \(\VMAD\) as the closure of the Dirichlet polynomials in \(\BMAD\), equipped with the \(\BMAD\) norm.

Let \(\D=\{z\in\mathbb C:|z|<1\}\), and let \(H(\D)\) denote
the space of analytic functions on \(\D\).  For \(g\in H(\D)\), the classical
Volterra-type operator \(J_g\) is defined by
\[
 J_gf(z)=\int_0^z f(\zeta)g'(\zeta)\dd\zeta,
\]
for \(z\in\D\) and \(f\in H(\D)\).  The discussion of the Volterra-type
operator \(J_g\) first arose in connection with semigroups of composition
operators.  Aleman and Siskakis \cite{AlemanSiskakis,AlemanSiskakisBergman} studied the boundedness and compactness of \(J_g\)
on classical Hardy and Bergman spaces. Subsequently,
Siskakis and Zhao \cite{SiskakisZhao} characterized its boundedness and compactness
on \(BMOA(\mathbb D)\) and \(VMOA(\mathbb D)\).   Laitila, Miihkinen and Nieminen later obtained estimates
for the essential and weak essential norms and proved that weak compactness and
compactness  coincide
\cite[Theorem~2]{LMN}.
For further details and background, see for example
\cite{MiihkinenPauPeralaWang,
Pommerenke,SiskakisSurvey,SiskakisZhao}.

In the Dirichlet series setting, given \(g\in\cD\), the Volterra operator
\(T_g\) is defined for Dirichlet series \(f\) by
\begin{equation}
T_gf(s)=-\int_s^{+\infty}f(w)g'(w)\dd w,
\end{equation}
where \(\Real s\) is large enough.  This operator was initially defined by
Brevig, Perfekt and Seip \cite{BPS}.  In their work, they gave some necessary
and sufficient conditions for the boundedness of \(T_g\) acting on Hardy
spaces of Dirichlet series.  More recently, the properties of Volterra operators
\(T_g\) on Bergman spaces of Dirichlet series and Bloch spaces of Dirichlet series were also investigated, see
\cite{BommierHato,ChenWang,GomezCabello}.  To the best of our knowledge, \(T_g\) has not previously been studied on \(BMOA_{\mathcal D}\) or \(VMOA_{\mathcal D}\).  In this paper, we continue this line of research by completely characterizing
the boundedness and compactness of Volterra operators \(T_g\) on
\(BMOA_{\mathcal D}\) and \(VMOA_{\mathcal D}\), and estimating their
essential and weak essential norms.

Write \(s=\sigma+i\tau\), and, for \(t\in\mathbb R\) and \(h>0\), set
\[
 Q(t,h)=\{\sigma+i\tau:0<\sigma<h,\ t<\tau<t+h\}.
\]
For \(g\in\BMAD\) and \(0<\delta\leq1\), set \(d\mu_g(s)=|g'(s)|^2\sigma\dd\sigma\,\dd\tau\) and
\[
 [g]_{\LMAD,\delta}
 :=\sup_{\substack{t\in\mathbb R\\0<h\leq\delta}}
 \log(e/h)\left(\frac{\mu_g(Q(t,h))}{h}\right)^{1/2}.
\]
Define
\[
 \LMAD
 =\left\{g\in\BMAD:[g]_{\LMAD}:=[g]_{\LMAD,1}<\infty\right\}.
\]
We equip \(\LMAD\) with the norm
$
 \|g\|_{\LMAD}=|g(+\infty)|+[g]_{\LMAD}.
$ We also set
\[
LMOA_{\mathcal D,0}
=\left\{g\in\LMAD:
\lim_{\delta\rightarrow0}[g]_{\LMAD,\delta}=0\right\}.
\]

Our first main result is the following.
\begin{samepage}
\begin{theorem}\label{thm:bounded}
Let \(g\in\cD\).  The following
conditions are equivalent:
\begin{description}
 \setlength{\itemsep}{0pt}
 \setlength{\parsep}{0pt}
 \setlength{\topsep}{2pt}
 \setlength{\partopsep}{0pt}
 \item[(i)] \(T_g:\BMAD\rightarrow\BMAD\) is bounded;
 \item[(ii)] \(T_g:\VMAD\rightarrow\VMAD\) is bounded;
 \item[(iii)] \(g\in\LMAD\).
\end{description}
Moreover, 
\[
 \|T_g\|_{\BMAD\to\BMAD}\asymp \|T_g\|_{\VMAD\to\VMAD} \asymp[g]_{\LMAD}.
\]
\end{theorem}
\end{samepage}

 For a bounded operator on a Banach space \(X\), let \(\|\cdot\|_{e,X}\) and
\(\|\cdot\|_{w,X}\) denote its distances from the compact and weakly compact
operators on \(X\), respectively.  Recall that an operator is completely
continuous if it maps every weakly convergent sequence to a norm convergent
sequence. Our second main result is the following.

\begin{theorem}\label{thm:essential}
Let \(g\in\LMAD\), and let \(X\) be either \(\BMAD\) or
\(\VMAD\). Then
\begin{equation}\label{eq:essential-norm}
\begin{aligned}
 \|T_g\|_{w,X}
 \asymp\|T_g\|_{e,X}&\asymp \operatorname{dist}_{\LMAD}
       (g,LMOA_{\mathcal D,0})\asymp
 \lim_{\delta\rightarrow0}[g]_{\LMAD,\delta}.
\end{aligned}
\end{equation}
The constants are independent of \(g\) and \(X\).
For each such \(X\), the following conditions are equivalent:
\begin{description}
 \setlength{\itemsep}{0pt}
 \setlength{\parsep}{0pt}
 \setlength{\topsep}{2pt}
 \setlength{\partopsep}{0pt}
 \item[(i)] \(T_g:X\to X\) is compact;
 \item[(ii)] \(T_g:X\to X\) is weakly compact;
 \item[(iii)] \(T_g:X\to X\) is completely continuous;
 \item[(iv)] \(g\in LMOA_{\mathcal D,0}\).
\end{description}
\end{theorem}

The compactness characterization in Theorem~\ref{thm:essential} yields the
following consequence for Dirichlet polynomial symbols.

\begin{corollary}\label{cor:polynomial}
Every Dirichlet polynomial symbol induces a compact Volterra operator on
both \(\BMAD\) and \(\VMAD\).
\end{corollary}

Motivated by Corollary~\ref{cor:polynomial}, we next ask whether every
Dirichlet polynomial symbol also induces a compact Volterra operator on
Bergman spaces of Dirichlet series.  This family of spaces was introduced in
\cite{BL} and further studied in  \cite{GomezCabello}.  We address the question at the endpoint \(p=1\) within the
framework of \cite{GomezCabello}.

Let \(\mu\) be an admissible measure, that is, a probability measure on
\((0,\infty)\) such that \(0\in\operatorname{supp}(\mu)\).  For
\(0<p<\infty\) and a Dirichlet polynomial \(P\), write
\(P_\sigma(s)=P(s+\sigma), \sigma\geq 0\) and set
\[
\|P\|_{\mathcal A_\mu^p}^p
=\int_0^\infty \|P_\sigma\|_{\mathcal H^p}^p\,\dd\mu(\sigma),
\]
where \(\mathcal H^p\) denotes the Hardy space of Dirichlet series.  The
space \(\mathcal A_\mu^p\) is defined as the completion of the Dirichlet
polynomials with respect to this norm, or quasi-norm when \(0<p<1\). 
G\'{o}mez-Cabello, Lef\`{e}vre and Queff\'{e}lec proved that every Dirichlet polynomial symbol induces a compact Volterra operator on
\(\mathcal A_\mu^p\), for every admissible measure \(\mu\), when
\(1<p<\infty\)
\cite[Proposition~7.12]{GomezCabello}.
For the McCarthy-Bergman spaces \(\mathcal A_\alpha^p\), Fu, Guo and Yan obtained compactness for all \(p\geq1\) under a little Bloch condition
\cite[Theorem~5.10(a)]{FuGuoYan}. Thus, the endpoint \(p=1\) was already known for this particular family of measures. Their proof is based on the compactness of horizontal translations and of the primitive operator
\cite[Lemma~5.9]{FuGuoYan}.

We prove that the polynomial symbol compactness result at \(p=1\) holds for every admissible measure \(\mu\). Our argument adapts the method of Fu, Guo and Yan to the general spaces \(\mathcal A_\mu^1\). The result therefore extends the theorem of G\'{o}mez-Cabello, Lef\`{e}vre and Queff\'{e}lec to the endpoint \(p=1\), while extending the polynomial symbol consequence of the endpoint result of Fu, Guo and Yan from the McCarthy-Bergman family to arbitrary admissible measures.
 See also the related results of Chen and Wang \cite{ChenWang}.  The following theorem settles the endpoint \(p=1\) for arbitrary admissible \(\mu\).

\begin{theorem}\label{thm:bergman-polynomial}
Let \(\mu\) be an admissible measure.  Every Dirichlet polynomial symbol
\(g\) induces a compact Volterra operator
$T_g$ on $\mathcal A_\mu^1.$
\end{theorem}

The paper is organized as follows.  Section~2 collects the preliminary
lemmas.  Section~3 proves the main results on \(\BMAD\) and \(\VMAD\), while
Section~4 establishes the endpoint compactness result on
\(\mathcal A_\mu^1\).  Throughout the paper, \(A\lesssim B\) means that
\(A\leq CB\) for some constant \(C>0\), and \(A\asymp B\) means that both
\(A\lesssim B\) and \(B\lesssim A\) hold.

\section{Preliminaries}
We begin by collecting the basic properties of spaces and several lemmas needed in the proofs.
Let \(\varphi(s)=(1-s)/(1+s)\), which maps \(\Czero\) conformally onto
\(\D\).  The conformally invariant Hardy space
\(H_{\mathrm{i}}^2(\Czero)\) consists of
those analytic functions \(f\) on \(\Czero\) for which
\(f\circ\varphi^{-1}\in H^2(\D)\), with
\[
 \|f\|_{H_{\mathrm{i}}^2(\Czero)}^2
 =\|f\circ\varphi^{-1}\|_{H^2(\D)}^2
 =\frac1\pi\int_{\mathbb R}|f(i\tau)|^2\frac{\dd\tau}{1+\tau^2}.
\]

For \(0<p<\infty\), the Hardy space of Dirichlet series \(\mathcal H^p\)
is defined as the completion of the Dirichlet polynomials
$P(s)=\sum_{n=1}^N a_n n^{-s}$ with respect to
\[
\|P\|_{\mathcal H^p}^p
=
\lim_{T\to\infty}\frac1{2T}
\int_{-T}^{T}|P(it)|^p\,\dd t.
\]
We write \(\mathcal H^\infty\) for the space of bounded Dirichlet series
on \(\Czero\), equipped with the norm
\[
\|f\|_{\mathcal H^\infty}
=
\sup_{\Real s>0}|f(s)|.
\]
See \cite{Bayart2002,HPK} for further background.

The following lemma can be found in \cite[Lemmas 3.4 and 3.6]{KQSS}.
Indeed, the inclusion
\(BMOA\subset\operatorname{Bloch}\) and Lemma~3.6 in \cite{KQSS} yield the
pointwise estimate, while the assertions on the abscissae follow from Bohr's theorem.  Here
\(\sigma_u(f)\) and \(\sigma_a(f)\) denote the abscissae of uniform and
absolute convergence, respectively.

\begin{lemma}\label{lem:box-growth}
Let \(f\in H_{\mathrm{i}}^2(\Czero)\cap\cD\). Then
\begin{description}
	\item[(i)]  $f$ belongs to $\BMAD$ if and only if $$ \sup_{t\in\mathbb R,\,0<h\leq1}\frac1h
	\int_{Q(t,h)}|f'(s)|^2\sigma\dd\sigma\,\dd\tau<\infty,$$ and in this case,
	\begin{align}\label{eq:bmo-carleson}
	\|f\|_{\operatorname{BMO}(\Czero)}^2
		\asymp
		\sup_{t\in\mathbb R,\,0<h\leq1}\frac1h
		\int_{Q(t,h)}|f'(s)|^2\sigma\dd\sigma\,\dd\tau.
	\end{align}
	\item[(ii)] If \(f\in\BMAD\), then
	for \(s=\sigma+i\tau\in\Czero\),
	\begin{equation}\label{eq:growth}
			|f(\sigma+i\tau)-f(+\infty)|
			\lesssim \bigl(\log^+(1/\sigma)+2^{-\sigma}\bigr)\|f\|_{\BMAD},
	\end{equation}
	and \(\sigma_u(f)\leq0\), \(\sigma_a(f)\leq\frac12\).
\end{description}
\end{lemma}

The following lemma will also be used in the compactness proof.

\begin{lemma}\label{lem:local-embedding}
Let \(g\in\LMAD\) and \(0<\delta\leq1\). Then
$$
 \int_{Q(t,h)}|f(s)|^2\dd\mu_g(s)
 \leq C [g]_{\LMAD,\delta}^2h\|f\|_{\BMAD}^2
$$
for every \(f\in\BMAD\), \(t\in\mathbb R\) and \(0<h\leq\delta\). 
The constant is independent of \(\delta\).
\end{lemma}

\begin{proof}
Fix \(Q:=Q(t,h)\), \(0<h\leq\delta\) and put \(K:=Q(0,1)\).
 For every Borel set \(E\subset\Czero\), define
\[
\nu(E)=\frac1h\mu_g\bigl(\{hw+it:w\in E\cap K\}\bigr).
\]
For \(x\in\mathbb R\) and \(0<r\leq1\), we have 
\[
 \nu(Q(x,r))
 \leq [g]_{\LMAD,\delta}^2\frac{r}{\log^2(e/(hr))}
 \leq [g]_{\LMAD,\delta}^2r
\]
and
\(
 \nu(K)\leq [g]_{\LMAD,\delta}^2.
\)
If \(r>1\), then, since \(\nu\) is supported on \(K\),
\[
 \nu(Q(x,r))\leq\nu(K)\leq [g]_{\LMAD,\delta}^2
 \leq r[g]_{\LMAD,\delta}^2.
\]
Thus \(\nu\) is a Carleson measure on \(\Czero\), with Carleson constant at
most \(C[g]_{\LMAD,\delta}^2\).

Set
$F(w)=f(hw+it)-f(h+it)$.
The translation and dilation invariance of the \(BMO\) seminorm
\cite[Chapter~VI, Section~1]{Garnett} gives
\(\|F\|_{\operatorname{BMO}(\Czero)}=\|f\|_{\operatorname{BMO}(\Czero)}\).
By the Poisson characterization of \(BMOA(\Czero)\)
\cite[Chapter~VI, Corollary~2.4]{Garnett}, and since \(F(1)=0\), we obtain
\begin{equation}\label{eq:weighted-bmo}
 \int_{\mathbb R}|F(iy)|^2\frac{\dd y}{1+y^2}
 \leq C\|F\|_{\operatorname{BMO}(\Czero)}^2.
\end{equation}
Put \(G(w)=F(w)/(1+w)\).  By
\eqref{eq:weighted-bmo}, we get \(G\in H^2(\Czero)\) and
$
 \|G\|_{H^2(\Czero)}^2 \leq C\|f\|_{\operatorname{BMO}(\Czero)}^2.
$
Since \(|1+w|\leq C\) on \(K\),
\cite[Chapter~II, Theorem~3.9]{Garnett} yields
\[
 \int_K|F(w)|^2\dd\nu(w)
 \leq C[g]_{\LMAD,\delta}^2
 \|f\|_{\operatorname{BMO}(\Czero)}^2.
\]
Therefore, making the change of variables $s=hw+it$ yields
\begin{align}\label{eq:local-oscillation}
 \int_Q|f(s)-f(h+it)|^2\dd\mu_g(s)
 \leq C[g]_{\LMAD,\delta}^2h
 \|f\|_{\operatorname{BMO}(\Czero)}^2.
\end{align}
Finally, \eqref{eq:growth} and the definition of
\([g]_{\LMAD,\delta}\) give
\[
 |f(h+it)|\leq C\log(e/h)\|f\|_{\BMAD},
\]
and
\[
 \mu_g(Q)\leq [g]_{\LMAD,\delta}^2
 \frac{h}{\log^2(e/h)}.
\]
Thus
\[
 |f(h+it)|^2\mu_g(Q)
 \leq C[g]_{\LMAD,\delta}^2h\|f\|_{\BMAD}^2.
\]
Together with \eqref{eq:local-oscillation}, this yields the desired result.
\end{proof}

We also need the following lemma.

\begin{lemma}\label{lem:distance}
For every \(g\in\LMAD\), one has
\begin{equation}\label{eq:distance}
 \operatorname{dist}_{\LMAD}(g,LMOA_{\mathcal D,0})
 \asymp \lim_{\delta\rightarrow0}[g]_{\LMAD,\delta}.
\end{equation}
\end{lemma}

\begin{proof}
If \(k\in LMOA_{\mathcal D,0}\), then Minkowski's inequality gives
\[
 [g]_{\LMAD,\delta}
 \leq [g-k]_{\LMAD}+[k]_{\LMAD,\delta}.
\]
Letting \(\delta\rightarrow0\) and then taking the infimum over \(k\) proves
the lower estimate in \eqref{eq:distance}.

Conversely, we first note that
\(g_\varepsilon(s)=g(s+\varepsilon), \varepsilon>0\) is an ordinary Dirichlet series.  By the
half-plane Poisson formula and semigroup property
\cite[Chapter~I, Section~3]{Garnett}, horizontal translation is contractive
on the \(BMOA\) seminorm.  Consequently,
$
 g_\varepsilon\in\BMAD,
$ and 
\[
 \|g_\varepsilon\|_{\operatorname{BMO}(\Czero)}
 \leq\|g\|_{\operatorname{BMO}(\Czero)}.
\]
The growth estimate
\eqref{eq:growth} and Cauchy's estimate show that \(g'\) is bounded on
\(\Real s\geq\varepsilon\).  Hence
\[
 \mu_{g_\varepsilon}(Q(t,h))
 \leq C_\varepsilon h^3,
\]
 and \(h^2\log^2(e/h)\rightarrow0\) as \(h\rightarrow0\).  Thus
\(g_\varepsilon\in LMOA_{\mathcal D,0}\).

Fix \(0<r<1/12\) and \(0<\varepsilon\leq r\).   If \(\varepsilon\leq h\leq r\), then the change of
variable \(u=\sigma+\varepsilon\) gives
\[
 \int_{Q(t,h)}|g'(s+\varepsilon)|^2\sigma
 \dd\sigma\,\dd\tau
 \leq \mu_g(Q(t,2h)).
\]
Since \(\log(e/h)\asymp\log(e/(2h))\),
$$
 \int_{Q(t,h)}|g'(s+\varepsilon)|^2\sigma
\dd\sigma\,\dd\tau
\leq \frac{C[g]_{\LMAD,2r}^2h}{\log^2(e/h)}.
$$

Suppose instead that \(0<h<\varepsilon\).  For
\(z=u+iv\), where \(\varepsilon\leq u<2\varepsilon\), write
\(w=\xi+i\eta\).  Subharmonicity on the disk \(D(z,u/4)\) yields
\begin{align*}
 |g'(z)|^2
 &\leq \frac{C}{u^2}\int_{D(z,u/4)}|g'(w)|^2
 \dd\xi\,\dd\eta\\
 &\leq \frac{C}{u^3}\mu_g(Q(v-u,3u))
 \leq \frac{C[g]_{\LMAD,6r}^2}
              {\varepsilon^2\log^2(e/\varepsilon)}.
\end{align*}
Indeed, \(D(z,u/4)\subset Q(v-u,3u)\) and
\(\Real w\geq3u/4\) on \(D(z,u/4)\).  Therefore
\[
 \int_{D(z,u/4)}|g'(w)|^2\dd\xi\,\dd\eta
 \leq\frac{C}{u}\int_{D(z,u/4)}|g'(w)|^2\Real w
 \dd\xi\,\dd\eta,
\]
which explains the second inequality above.    It follows that
\[
 \int_{Q(t,h)}|g'(s+\varepsilon)|^2\sigma
 \dd\sigma\,\dd\tau
 \leq \frac{C[g]_{\LMAD,6r}^2h^3}
              {\varepsilon^2\log^2(e/\varepsilon)}.
\]
Since \(0<h<\varepsilon\), we know that
\[
 \left(\frac{h}{\varepsilon}\right)^2
 \frac{\log^2(e/h)}{\log^2(e/\varepsilon)}\leq C.
\]
Combining these estimates, for \(0<\varepsilon\leq r<1/12\), we get
\begin{equation}\label{eq:translated-small-boxes}
 [g_\varepsilon]_{\LMAD,r}
 \leq C[g]_{\LMAD,6r}.
\end{equation}

By our assumption  \(g\in\LMAD\), we know that
\[
 \sup_{t\in\mathbb R,\,0<h\leq\delta}
 \frac{\mu_g(Q(t,h))}{h}
 \leq \frac{[g]_{\LMAD}^2}{\log^2(e/\delta)}
 \rightarrow0, \quad \delta \rightarrow 0.
\]
Together with
\cite[Chapter~VI, Exercise~14(b)]{Garnett}, this proves
\(g\in VMOA(\Czero)\cap\cD\).  The approximation characterization
\cite[Chapter~VI, Theorem~5.1(c)]{Garnett} then gives
\(\|g-g_\varepsilon\|_{\operatorname{BMO}(\Czero)}\to0\).  For the boxes with
\(r<h\leq1\), Lemma~\ref{lem:box-growth} consequently gives
\begin{align*}
 &\sup_{t,\,r<h\leq1}\log^2(e/h)\frac1h
 \int_{Q(t,h)}|g'(s)-g'(s+\varepsilon)|^2\sigma
 \dd\sigma\,\dd\tau\\
 \leq&
 C\log^2(e/r)\|g-g_\varepsilon\|_{\operatorname{BMO}(\Czero)}^2
 \rightarrow0.
\end{align*}
Consequently, Minkowski's inequality and
\eqref{eq:translated-small-boxes} give
\[
 [g-g_\varepsilon]_{\LMAD,r}
 \leq [g]_{\LMAD,r}+[g_\varepsilon]_{\LMAD,r}
 \leq C[g]_{\LMAD,6r}.
\]
Since \(g-g_\varepsilon\) has zero constant coefficient, we conclude that
\[
 \limsup_{\varepsilon\rightarrow0}
 \|g-g_\varepsilon\|_{\LMAD}
 \leq C[g]_{\LMAD,6r}.
\]
Letting \(r\rightarrow0\), and recalling that
\(g_\varepsilon\in LMOA_{\mathcal D,0}\), we obtain
\[
 \operatorname{dist}_{\LMAD}(g,LMOA_{\mathcal D,0})
 \leq \limsup_{\varepsilon\rightarrow0}
       \|g-g_\varepsilon\|_{\LMAD}
 \leq C\lim_{\delta\rightarrow0}[g]_{\LMAD,\delta}.
\]
This proves the upper estimate in \eqref{eq:distance}.
\end{proof}

The following lemma can be found in
\cite[Proposition~6]{LNST}, see also \cite{Leibov}.

\begin{lemma}\label{lem:c0}
	Let \((v_n)\subset VMOA(\D)\) have \(BMOA\) seminorms bounded above
	and away from zero, and suppose that
$
	\|v_n\|_{H^2(\D)}\longrightarrow0.
$
	Then there exist a subsequence \((v_{n_j})\) and constants
	\(0<c\leq C<\infty\) such that
	\[
	c\sup_j|\lambda_j|
	\leq
	\left\|\sum_j\lambda_jv_{n_j}\right\|_{BMOA(\D)}
	\leq
	C\sup_j|\lambda_j|
	\]
	for every finitely supported scalar sequence \((\lambda_j)\). Consequently,
	\[
	\iota:c_0\longrightarrow
	\overline{\operatorname{span}}\{v_{n_j}:j\geq1\},
	\qquad
	\iota((\lambda_j))
	=\sum_{j=1}^{\infty}\lambda_jv_{n_j},
	\]
	is an isomorphism onto its range.  In particular,
	\((v_{n_j})\) is equivalent to the canonical basis of \(c_0\).
\end{lemma}

\section{Proofs of the main results}
In this section, we prove our two main results,
Theorems~\ref{thm:bounded} and~\ref{thm:essential}.
We begin with a lemma that connects functions in \(BMOA(\D)\) with
Dirichlet series supported on powers of \(2\). We then use it to construct
the test functions needed below.

\begin{lemma}\label{lem:periodic-composition}
Let \(H\in BMOA(\D)\), with
\(H(z)=\sum_{m\geq0}c_mz^m\), and set
\[
 H(2^{-s})=\sum_{m\geq0}c_m(2^m)^{-s}.
\]
Then \(H(2^{-s})\in\BMAD\), and
\[
 \|H(2^{-s})\|_{\BMAD}\asymp\|H\|_{BMOA(\D)}.
\]
\end{lemma}

\begin{proof}
Under the Bohr correspondence, a Dirichlet series supported on \(\{2^m:m\geq0\}\)
has a Bohr lift depending only on the first variable \(z_1\), in particular,
the Bohr lift of \(H(2^{-s})\) is \(H(z_1)\). Its boundary function
\(H(e^{-it\log2})\) is periodic with period \(2\pi/\log2\).  Hence the standard identification of circle \(BMO\) with
periodic \(BMO\) on \(\mathbb R\) (see \cite[p.~138]{KQSS}) shows that the corresponding \(BMO\)
seminorms are comparable. Moreover, since
\(BMOA(\D)\subset H^2(\D)\), splitting \(\mathbb R\) into intervals of length \(\frac{2\pi}{\log2}\)
and using the periodicity, we have
\[
\int_{\mathbb R}
|H(e^{-it\log2})|^2\frac{\dd t}{1+t^2}
\lesssim
\int_0^{2\pi/\log2}|H(e^{-it\log2})|^2\,\dd t
\asymp
\|H\|_{H^2(\D)}^2<\infty.
\]
Thus \(H(2^{-s})\in H_{\mathrm{i}}^2(\Czero)\), and consequently
\(H(2^{-s})\in\BMAD\). Since the constant term of \(H(2^{-s})\) is
\(H(0)\), the norm equivalence follows.
\end{proof}

\begin{lemma}\label{lem:tests}
	For each \(a\in\D\), define \(H_a\in H(\D)\) by
	\[
	H_a(z)=\log\frac{1}{1-\overline a z},
	\qquad z\in\D.
	\]
	Then the following assertions hold.
	\begin{description}
		\item[(i)] For every \(a\in\D\), the function
		\[
		s\longmapsto H_a(2^{-s}),
		\qquad s\in\Czero,
		\]
		belongs to \(\VMAD\). Moreover,
		$
		\sup_{a\in\D}
		\|H_a(2^{-s})\|_{\BMAD}<\infty.
	$
		
		\item[(ii)] There exist constants \(h_0\in(0,1)\) and \(c>0\) such that,
		for every \(t\in\mathbb R\) and \(0<h\leq h_0\), if
	$
		a=(1-h)e^{-i(t+h/2)\log2},
	$
		then \(a\in\D\) and
		\begin{equation}\label{eq:test-lower}
			|H_a(2^{-s})|
			\geq c\log(e/h),
			\qquad s\in Q(t,h).
		\end{equation}
	\end{description}
\end{lemma}

\begin{proof}
We first prove (i).
For each fixed \(a\in\D\), we have
\[
\left\|
H_a-\sum_{m=1}^N\frac{\overline a^{\,m}z^m}{m}
\right\|_{BMOA(\D)}
\lesssim
\sum_{m=N+1}^{\infty}\frac{|a|^m}{m}
\rightarrow0.
\] Thus, by Lemma~\ref{lem:periodic-composition},
\[
\left\|
\sum_{m=1}^N\frac{\overline a^{\,m}}{m}(2^m)^{-s}
-H_a(2^{-s})
\right\|_{\BMAD}
\rightarrow0.
\]
Each partial sum on the left is a Dirichlet polynomial. Therefore,
\(H_a(2^{-s})\) belongs to the closure of the Dirichlet polynomials in
\(\BMAD\), and hence
$
H_a(2^{-s})\in\VMAD.
$ Moreover,
applying Lemma~\ref{lem:periodic-composition} once more, we obtain
\[
\sup_{a\in\D}
\|H_a(2^{-s})\|_{\BMAD}\asymp\sup_{a\in\D}\|H_a\|_{BMOA(\D)}<\infty.
\]

Next, we prove (ii).
Let \(t\in\mathbb R\), \(0<h\leq h_0\), and
$
a=(1-h)e^{-i(t+h/2)\log2}.
$
For \(s=\sigma+i\tau\in Q(t,h)\), we have
\[
\overline a\,2^{-s}
=(1-h)2^{-\sigma}
e^{-i(\tau-t-h/2)\log2}.
\]
Using \(1-2^{-\sigma}\leq\sigma\log2\) and
\(|1-e^{-ix}|\leq|x|\), we obtain
\[
\begin{aligned}
	|1-\overline a\,2^{-s}|
	&\leq
	1-(1-h)2^{-\sigma}
	+(1-h)2^{-\sigma}
	\big|1-e^{-i(\tau-t-h/2)\log2}\big|\\
	&\leq
	h+\sigma\log2
	+|\tau-t-h/2|\log2\\
	&\leq
	\left(1+\frac32\log2\right)h.
\end{aligned}
\]
Since $ \Real(1-\overline a\,2^{-s}) \geq1-|a|2^{-\sigma}>0$, we may
choose \(h_0>0\) sufficiently small such that, for \(0<h\leq h_0\) and
\(s\in Q(t,h)\),
\[
|H_a(2^{-s})|
\geq\Real H_a(2^{-s})
=\log\frac{1}{|1-\overline a\,2^{-s}|}
\geq
\log\frac{1}
{\left(1+\frac32\log2\right)h}\geq\frac12\log(e/h).
\]
Thus \eqref{eq:test-lower} holds with \(c=1/2\).
\end{proof}
 
We now prove Theorem \ref{thm:bounded}.
\begin{proof}[Proof of Theorem~\ref{thm:bounded}]
We first prove (iii)\(\Rightarrow\)(i).
Suppose that \(g\in\LMAD\) and \(f\in\BMAD\).
Fix \(0<\delta\leq1\).  If \(0<h\leq\delta\),
Lemma~\ref{lem:local-embedding} gives
\[
 \int_{Q(t,h)}|f(s)|^2\dd\mu_g(s)
 \leq C[g]_{\LMAD,\delta}^2h\|f\|_{\BMAD}^2.
\]
If \(h>\delta\), choose $N\in\mathbb N$ such that $(N-1)\delta<h\leq N\delta$.  Then
\[
 Q(t,h)\cap\{0<\sigma<\delta\}
 \subset\bigcup_{k=0}^{N-1}Q(t+k\delta,\delta),
\]
where \(N\delta\leq h+\delta\leq2h\).  Applying
Lemma~\ref{lem:local-embedding} yields
\[
\begin{aligned}
 \int_{Q(t,h)\cap\{0<\sigma<\delta\}}
 |f(s)|^2\dd\mu_g(s)
 &\leq C N\delta[g]_{\LMAD,\delta}^2\|f\|_{\BMAD}^2\\
 &\leq C[g]_{\LMAD,\delta}^2h\|f\|_{\BMAD}^2.
\end{aligned}
\]
Consequently,
\begin{equation}\label{eq:strip-embedding}
 \sup_{\substack{t\in\mathbb R\\h>0}}
 \frac1h\int_{Q(t,h)\cap\{0<\sigma<\delta\}}
 |f(s)|^2\dd\mu_g(s)
 \lesssim[g]_{\LMAD,\delta}^2\|f\|_{\BMAD}^2.
\end{equation}
Taking \(\delta=1\) in \eqref{eq:strip-embedding}, it remains to control
\(Q(t,h)\cap\{\sigma\geq1\}\). Lemma~\ref{lem:box-growth} gives directly
\begin{equation}\label{eq:g-bmo-M1}
 \|g\|_{\operatorname{BMO}(\Czero)}^2
 \lesssim
 \sup_{t\in\mathbb R,\,0<h\leq1}\frac{\mu_g(Q(t,h))}{h}
 \leq[g]_{\LMAD}^2.
\end{equation}
If \(\sigma\geq3/2\), apply \eqref{eq:growth} to \(g-g(+\infty)\) on the circle
\(D(\sigma+i\tau,1/2)\) and then use Cauchy's estimate.  For
\(1\leq\sigma\leq3/2\), use the Bloch estimate
\(\sigma|g'(\sigma+i\tau)|\lesssim
\|g\|_{\operatorname{BMO}(\Czero)}\).  Together with
\eqref{eq:g-bmo-M1}, for
\(\sigma\geq1\) and \(\tau\in\mathbb R\),
\begin{equation}
 |g'(\sigma+i\tau)|\lesssim
 2^{-\sigma}[g]_{\LMAD}.
\end{equation}
The growth estimate for \(f\) also gives, for \(\sigma\geq1\),
$
 |f(\sigma+i\tau)|\lesssim\|f\|_{\BMAD}.
$
It follows that
\begin{equation}\label{eq:Tgf-tail}
 \sup_{\tau\in\mathbb R}\int_1^\infty
 |f(\sigma+i\tau)g'(\sigma+i\tau)|^2\sigma\dd\sigma
 \lesssim[g]_{\LMAD}^2\|f\|_{\BMAD}^2.
\end{equation}
In addition, integrating along the positive real axis gives
\begin{equation}\label{eq:Tgf-basepoint}
 |(T_gf)(1)|
 \leq\int_1^\infty|f(\sigma)g'(\sigma)|\dd\sigma
 \lesssim[g]_{\LMAD}\|f\|_{\BMAD}.
\end{equation}
Thus, \eqref{eq:Tgf-tail} and Fubini's theorem yield
\[
 \begin{aligned}
 \int_{Q(t,h)\cap\{\sigma\geq1\}}
   |f(s)g'(s)|^2\sigma\dd\sigma\,\dd\tau
 &=\int_t^{t+h}\int_1^h
   |f(\sigma+i\tau)g'(\sigma+i\tau)|^2
   \sigma\dd\sigma\dd\tau\\
 &\leq C[g]_{\LMAD}^2h\|f\|_{\BMAD}^2.
 \end{aligned}
\]
Combining these estimates, we get
\[
 \sup_{\substack{t\in\mathbb R\\h>0}}
 \frac1h\int_{Q(t,h)}|f(s)g'(s)|^2
 \sigma\dd\sigma\,\dd\tau
 \leq C[g]_{\LMAD}^2\|f\|_{\BMAD}^2.
\]
By \cite[Chapter~VI, Exercise~6(b)]{Garnett} and
\eqref{eq:Tgf-basepoint}, \(T_gf\in BMOA(\Czero)\).  Since \(T_gf\in\cD\)
and \((T_gf)(+\infty)=0\), we conclude that
\[
 \|T_gf\|_{\BMAD}
 \leq C[g]_{\LMAD}\|f\|_{\BMAD}.
\]

Now, we prove (iii)\(\Rightarrow\)(ii).
Let \(g\in\LMAD\) and let \(P\) be a Dirichlet polynomial.
By the implication (iii)\(\Rightarrow\)(i) established above,
\(T_gP\in\BMAD\), and
\[
 \sup_{\substack{t\in\mathbb R\\0<h\leq\delta}}
 \frac1h\int_{Q(t,h)}|(T_gP)'(s)|^2\sigma
 \dd\sigma\,\dd\tau
 \leq
 \frac{\|P\|_{\infty,\Czero}^2[g]_{\LMAD}^2}
 {\log^2(e/\delta)}
 \rightarrow0
\]
as \(\delta\rightarrow0\).
By 
\cite[Chapter~VI, Exercise~14(b)]{Garnett},
\(T_gP\in VMOA(\Czero)\) and it is an
ordinary Dirichlet series.  Hence \(T_gP\in VMOA(\Czero)\cap\cD=\VMAD\).

For \(f\in\VMAD\), Theorem~7.1 of \cite{BPS} gives Dirichlet polynomials
\((P_j)\) such that
$
 \|P_j-f\|_{\BMAD}\rightarrow0.
$
The case (iii)\(\Rightarrow\)(i) gives
\[
 \|T_gP_j-T_gf\|_{\BMAD}
 \leq C[g]_{\LMAD}\|P_j-f\|_{\BMAD}
 \rightarrow0.
\]
For every \(j\), 
$
 T_gP_j\in\VMAD.
$
Since \(\VMAD\) is closed in \(\BMAD\), we know that
$
 T_gf\in\VMAD.
$
Moreover,
\[
\begin{aligned}
 \|T_gf\|_{\VMAD}
 =\|T_gf\|_{\BMAD}
 \leq C[g]_{\LMAD}\|f\|_{\VMAD}.
\end{aligned}
\]
Therefore, we have
$
 \|T_g\|_{\VMAD\to\VMAD}\lesssim[g]_{\LMAD}.
$

Finally, we prove (i)\(\Rightarrow\)(iii) and (ii)\(\Rightarrow\)(iii). Let
\(X\) be either \(\BMAD\) or \(\VMAD\), and suppose that
\(T_g:X\to X\) is bounded.
Since \(1\in\VMAD\subset X\), then
$
 g-g(+\infty)=T_g1\in X\subset\BMAD
$
and
\[
 \|g\|_{\operatorname{BMO}(\Czero)}
 =\|T_g1\|_X
 \leq\|T_g\|_{X\to X}.
\]
In particular, \(g\in\BMAD\).

Fix \(t\in\mathbb R\) and \(0<h\leq h_0\), and choose \(a\) as in
Lemma~\ref{lem:tests}(ii). By Lemma~\ref{lem:tests}(i),
\(H_a(2^{-s})\in\VMAD\subset X\), and these test functions have uniformly
bounded \(X\)-norm. Since
\[
 \big(T_g(H_a(2^{-s}))\big)'=H_a(2^{-s})g',
\]
\eqref{eq:bmo-carleson} and \eqref{eq:test-lower} give
\[
\begin{aligned}
 \log^2(e/h)\frac{\mu_g(Q(t,h))}{h}
 \lesssim \|T_g(H_a(2^{-s}))\|_X^2\lesssim \|T_g\|_{X\to X}^2.
\end{aligned}
\]
For \(h_0<h\leq1\), applying \eqref{eq:bmo-carleson} to
\(T_g1=g-g(+\infty)\), and using
\(\log(e/h)\leq\log(e/h_0)\), gives the same estimate. Consequently,
$
 [g]_{\LMAD}\lesssim\|T_g\|_{X\to X}<\infty.
$
Thus \(g\in\LMAD\). The proof is complete.
\end{proof}

Next, we  prove Theorem \ref{thm:essential}.
\begin{proof}[Proof of Theorem~\ref{thm:essential}]
We first prove (iv)\(\Rightarrow\)(i).
Suppose that \(g\in LMOA_{\mathcal D,0}\).  By
Theorem~\ref{thm:bounded}, \(T_g\) is bounded.  Let \((f_j)\) be a bounded sequence in \(\BMAD\).  By
\eqref{eq:growth}, \((f_j)\) is locally uniformly bounded on
\(\Czero\).  Hence Montel's theorem yields a subsequence, still denoted by
\((f_j)\), that converges uniformly on compact subsets of \(\Czero\) to
some \(f\in H(\Czero)\).  In particular,
$
f_j(1)\rightarrow f(1).
$
Moreover, Cauchy's integral formula gives
$
f_j'\rightarrow f'
$
uniformly on compact subsets of \(\Czero\). By \cite[Chapter~VI, Exercise~6(b)]{Garnett}, we have
\[
\sup_j\sup_{\substack{t\in\mathbb R\\0<h\leq1}}
\frac1h\int_{Q(t,h)}
|f_j'(s)|^2\sigma\,\dd\sigma\,\dd\tau
\lesssim
\sup_j\|f_j\|_{\BMAD}^2
<\infty.
\]
Hence, Fatou's lemma gives
$
f\in BMOA(\Czero)
$
and
\[
\|f\|_{BMOA(\Czero)}
\leq C\sup_j\|f_j\|_{\BMAD}.
\]
It remains to show that \(f\in\cD\).  Since \(f_j\in\BMAD\), we have
\(\sigma_u(f_j)\leq0\), and \eqref{eq:growth} gives
\[
\sup_j
\|f_j(\,\cdot+1/2)\|_{\mathcal H^\infty}
\lesssim
\sup_j\|f_j\|_{\BMAD}
<\infty.
\]
Bayart's Montel theorem \cite[Lemma~18]{Bayart2002} therefore yields a
further subsequence converging uniformly on every half-plane
\(\Real s\geq\varepsilon>0\) to a function represented by an ordinary
Dirichlet series.  On the other hand, the same subsequence converges
locally uniformly to \(f(\,\cdot+1/2)\).  Hence, by uniqueness of the
limit, \(f(\,\cdot+1/2)\) is represented by an ordinary Dirichlet series.
 Thus \(f\in\cD\), and consequently \(f\in\BMAD\).

Put \(u_j=f_j-f\).  Then \((u_j)\) is bounded in \(\BMAD\) and
converges to zero locally uniformly on \(\Czero\).  For every \(\delta>0\),
\eqref{eq:growth} shows that
\((u_j(\,\cdot+\delta/2))\) is bounded in \(\mathcal H^\infty\).
Hence, by Lemma~18 in \cite{Bayart2002},
 every
subsequence of \((u_j)\) has a further subsequence converging uniformly
on \(\Real s\geq\delta\). Since \(u_j\to0\) locally uniformly on
\(\Czero\), the limit of every such subsequence is zero.  Consequently,
\[
 \sup_{\Real s\geq\delta}|u_j(s)|\rightarrow0.
\]

Given \(\varepsilon>0\), choose \(0<\delta<1\) such that
\(
 [g]_{\LMAD,\delta}^2<\varepsilon.
\)
Applying \eqref{eq:strip-embedding} with \(f=u_j\), we obtain
\[
 \sup_{\substack{t\in\mathbb R\\0<h\leq1}}
 \frac1h\int_{Q(t,h)\cap\{0<\sigma<\delta\}}
 |u_j(s)|^2\dd\mu_g(s)
 \lesssim[g]_{\LMAD,\delta}^2\|u_j\|_{\BMAD}^2.
\]
Since \(g\in\BMAD\), Lemma~\ref{lem:box-growth} also gives
\[
\begin{aligned}
 \sup_{\substack{t\in\mathbb R\\0<h\leq1}}
 \frac1h\int_{Q(t,h)\cap\{\sigma\geq\delta\}}
 |u_j(s)|^2\dd\mu_g(s)
 &\leq
 \sup_{\Real s\geq\delta}|u_j(s)|^2
 \sup_{\substack{t\in\mathbb R\\0<h\leq1}}
 \frac{\mu_g(Q(t,h))}{h}\\
 &\lesssim
 \|g\|_{\operatorname{BMO}(\Czero)}^2
 \sup_{\Real s\geq\delta}|u_j(s)|^2.
\end{aligned}
\]
Combining the last two estimates, we obtain
\[
\begin{aligned}
 &\sup_{\substack{t\in\mathbb R\\0<h\leq1}}
 \frac1h\int_{Q(t,h)}|u_j(s)|^2\dd\mu_g(s)\lesssim
 \varepsilon\|u_j\|_{\BMAD}^2
 +\|g\|_{\operatorname{BMO}(\Czero)}^2
 \sup_{\Real s\geq\delta}|u_j(s)|^2.
\end{aligned}
\]
Since \(\varepsilon>0\) is arbitrary, Lemma~\ref{lem:box-growth} yields
\[
\|T_gu_j\|_{\BMAD}^2\lesssim\sup_{\substack{t\in\mathbb R\\0<h\leq1}}
 \frac1h\int_{Q(t,h)}|u_j(s)|^2\dd\mu_g(s)
 \rightarrow0.
\]
Hence \(T_g\) is compact on \(\BMAD\).
By Theorem~\ref{thm:bounded},
\(T_g(\VMAD)\subset\VMAD\).  Since \(\VMAD\) is a closed subspace of
\(\BMAD\) and \(T_g\) is compact on \(\BMAD\), its restriction to
\(\VMAD\) is also compact.

The preceding argument also yields the upper estimate in
\eqref{eq:essential-norm}. In fact, if
\(k\in LMOA_{\mathcal D,0}\), the preceding argument
shows that \(T_k\) is compact on \(\BMAD\), and hence also on \(\VMAD\) after
restriction.  Since the constant functions belong to
\(LMOA_{\mathcal D,0}\), the distance can be computed using the seminorm:
\[
 \operatorname{dist}_{\LMAD}(g,LMOA_{\mathcal D,0})
 =\inf_{k\in LMOA_{\mathcal D,0}}[g-k]_{\LMAD}.
\]
Theorem~\ref{thm:bounded} and Lemma~\ref{lem:distance} therefore give, for
\(X\in\{\BMAD,\VMAD\}\),
\begin{equation}\label{eq:essential-upper}
\begin{aligned}
 \|T_g\|_{w,X}
 &\leq \|T_g\|_{e,X}\leq \inf_{k\in LMOA_{\mathcal D,0}}
       \|T_{g-k}\|_{X\to X}\\
 &\leq C\inf_{k\in LMOA_{\mathcal D,0}}[g-k]_{\LMAD}=C\operatorname{dist}_{\LMAD}(g,LMOA_{\mathcal D,0})\\
 &\leq C\lim_{\delta\rightarrow0}[g]_{\LMAD,\delta}.
\end{aligned}
\end{equation}

We next prove
(i)\(\Rightarrow\)(ii)\(\Rightarrow\)(iv).
Since every compact operator is weakly compact,
(i)\(\Rightarrow\)(ii) is immediate.
To prove (ii)\(\Rightarrow\)(iv) and the reverse estimate in
\eqref{eq:essential-norm}, put
\[
 \alpha=\lim_{\delta\rightarrow0}[g]_{\LMAD,\delta}.
\]
If \(\alpha=0\), Lemma~\ref{lem:distance} gives
\(g\in LMOA_{\mathcal D,0}\), and the reverse estimate is immediate.
Assume henceforth that \(\alpha>0\).
Choose \(h_n\rightarrow0\) satisfying \(0<h_n\leq h_0\), and choose
\(t_n\in\mathbb R\) so that
\begin{equation}\label{eq:alpha-boxes}
 \log(e/h_n)
 \left(\frac{\mu_g(Q(t_n,h_n))}{h_n}\right)^{1/2}
 \geq\frac{\alpha}{2}.
\end{equation}
Set
$
 a_n=(1-h_n)e^{-i(t_n+h_n/2)\log2}.
$
Then set
\[
 f_n(s)=H_{a_n}(2^{-s}).
\]
By Lemma~\ref{lem:tests} and \eqref{eq:alpha-boxes}, we obtain that \((f_n)\) is bounded in \(\BMAD\)
and
\begin{equation}\label{eq:own-box}
 \frac1{h_n}\int_{Q(t_n,h_n)}|f_n|^2\dd\mu_g
 \geq c\alpha^2.
\end{equation}

After passing to a subsequence, we may assume that \((a_n)\) is Cauchy.
For each fixed \(n\), the function \(f_n\) is bounded on \(\Czero\). Indeed, using
\[
H_{a_n}(z)
=\sum_{m=1}^{\infty}\frac{\overline{a_n}^{\,m}z^m}{m},
\]
we obtain, for \(s\in\Czero\),
\[
\begin{aligned}
|f_n(s)|
=|H_{a_n}(2^{-s})|\leq
\sum_{m=1}^{\infty}\frac{|a_n|^m}{m}
=\log\frac{1}{1-|a_n|}
=\log\frac1{h_n}<\infty.
\end{aligned}
\]
Hence, for every fixed \(n\),
\[
\begin{aligned}
 \sup_{\substack{t\in\mathbb R\\0<h\leq\delta}}
 \frac1h\int_{Q(t,h)}|f_n|^2\dd\mu_g
 &\leq
 \|f_n\|_\infty^2
 \sup_{\substack{t\in\mathbb R\\0<h\leq\delta}}
 \frac{\mu_g(Q(t,h))}{h}\\
 &\leq
 \frac{\|f_n\|_\infty^2[g]_{\LMAD}^2}
 {\log^2(e/\delta)}
 \rightarrow0,
 \qquad \delta\to0.
\end{aligned}
\]
Fix \(0<\theta<1\). We now choose a further subsequence as follows.
Suppose that \(n_j\) has
been chosen. Since the above limit holds for the fixed function \(f_{n_j}\),
there exists \(\delta_j>0\) such that
\[
 \sup_{\substack{t\in\mathbb R\\0<h\leq\delta_j}}
 \frac1h\int_{Q(t,h)}|f_{n_j}|^2\dd\mu_g
 \leq c\theta^2\alpha^2.
\]
Since \(h_n\to0\), we may choose \(n_{j+1}>n_j\) so that
\(h_{n_{j+1}}\leq\delta_j\). It follows that
\[
 \frac1{h_{n_{j+1}}}
 \int_{Q(t_{n_{j+1}},h_{n_{j+1}})}
 |f_{n_j}|^2\dd\mu_g
 \leq c\theta^2\alpha^2.
\]
Reindexing the subsequence, we may assume that
\begin{equation}\label{eq:previous-box}
 \frac1{h_{n+1}}
 \int_{Q(t_{n+1},h_{n+1})}
 |f_n|^2\dd\mu_g
 \leq c\theta^2\alpha^2.
\end{equation}

Let \(u_n=f_{n+1}-f_n\).  We apply the reverse triangle inequality in
\[
 L^2\left(Q(t_{n+1},h_{n+1}),\frac{\mu_g}{h_{n+1}}\right).
\]
Together with \eqref{eq:own-box} and \eqref{eq:previous-box}, this gives
\[
 \left(\frac1{h_{n+1}}\int_{Q(t_{n+1},h_{n+1})}
 |u_n|^2\dd\mu_g\right)^{1/2}
 \geq \sqrt{c}\,\alpha-\theta\sqrt{c}\,\alpha
 =(1-\theta)\sqrt{c}\,\alpha.
\]
Hence, Lemma~\ref{lem:box-growth} gives
\begin{equation}\label{eq:image-lower}
 \|T_gu_n\|_{\BMAD}\gtrsim\alpha.
\end{equation}
Moreover, \(u_n(+\infty)=0\), and the boundedness of \(T_g\), together
with \eqref{eq:image-lower}, gives
\[
 \inf_n\|u_n\|_{\operatorname{BMO}(\Czero)}
 =\inf_n\|u_n\|_{\BMAD}
 >0.
\]
Since \((f_n)\) is uniformly bounded in \(\BMAD\) and
\(u_n=f_{n+1}-f_n\),
\begin{align}\label{eq:un-bounded}
 \sup_n\|u_n\|_{\BMAD}
 \leq
 2\sup_n\|f_n\|_{\BMAD}<\infty.
\end{align}
Thus \((u_n)\) is bounded in \(\BMAD\).

Write
\[
 v_n(z)=H_{a_{n+1}}(z)-H_{a_n}(z).
\]
Then
$
 u_n(s)=v_n(2^{-s}).
$
Each \(v_n\) belongs to \(VMOA(\D)\).  Since \((a_n)\) is Cauchy,
dominated convergence gives
\[
 \|v_n\|_{H^2(\D)}^2
 =\sum_{m\geq1}
   \frac{|\overline{a_{n+1}}^{\,m}-\overline{a_n}^{\,m}|^2}{m^2}
 \rightarrow0.
\]
By Lemma~\ref{lem:c0}, a subsequence of \((v_n)\) is equivalent to the
canonical basis of \(c_0\). By Lemma~\ref{lem:periodic-composition}, the
corresponding subsequence of \((u_n)\) has the same property. Let \((e_n)\)
denote the canonical basis of \(c_0\). Hence, after reindexing, there exists an
isomorphism
\[
\iota:c_0\longrightarrow
\overline{\operatorname{span}}\{u_n:n\ge1\}\subset\BMAD
\]
such that
$
\iota(e_n)=u_n, n\ge1.
$

Lemma~\ref{lem:tests} (i) gives \(u_n\in\VMAD\). Since
\(c_{00}\) is dense in \(c_0\), \(\iota\) is bounded, and \(\VMAD\) is
closed in \(\BMAD\), we know that
\[
 \iota(c_{00})\subset\VMAD
 \quad\Longrightarrow\quad
 \iota(c_0)\subset\VMAD.
\]
Thus, for every \(X\in\{\BMAD,\VMAD\}\), the map
\(\iota:c_0\to X\) is an isomorphism onto its range.
Let \(S:X\to X\) be weakly compact. Then
\(S\circ\iota:c_0\to X\) is weakly compact. The canonical basis
\((e_n)\) is weakly null in \(c_0\), since, for every
\((\lambda_j)\in c_0^*=\ell^1\),
$
 \langle e_n,(\lambda_j)\rangle=\lambda_n\rightarrow0.
$
As in Lemma~7 of \cite{LMN}, the Dunford--Pettis property of \(c_0\)
therefore yields
\[
 \|Su_n\|_X=\|(S\circ\iota)e_n\|_X\rightarrow0.
\]
The norm of \(\VMAD\) is inherited from \(\BMAD\). Hence
\eqref{eq:image-lower}, \eqref{eq:un-bounded}, and the reverse triangle
inequality give
\[
 \liminf_{n\rightarrow\infty}\|(T_g-S)u_n\|_X\gtrsim\alpha
 \quad \text{and} \quad
 \sup_n\|u_n\|_X\lesssim1.
\]
Consequently,
\[
 \|T_g-S\|_{X\to X}
 \geq\limsup_{n\rightarrow\infty}
 \frac{\|(T_g-S)u_n\|_X}{\|u_n\|_X}
 \gtrsim\alpha.
\]
Taking the infimum over all weakly compact operators \(S:X\to X\), we obtain
\[
 \|T_g\|_{w,X}\gtrsim\alpha,
 \quad X\in\{\BMAD,\VMAD\},
\]
which proves the reverse estimate in \eqref{eq:essential-norm} for $X$.
Together with
\eqref{eq:essential-upper} and Lemma~\ref{lem:distance}, this proves
\eqref{eq:essential-norm}.  If (ii) holds, then
\(\|T_g\|_{w,X}=0\), and the reverse estimate implies \(\alpha=0\).
Lemma~\ref{lem:distance} therefore gives
\(g\in LMOA_{\mathcal D,0}\).  Thus (ii)\(\Rightarrow\)(iv).

Finally, we prove
(i)\(\Rightarrow\)(iii)\(\Rightarrow\)(iv).
Compact operators are completely continuous, so
(i)\(\Rightarrow\)(iii).  To prove (iii)\(\Rightarrow\)(iv), suppose that
\(g\notin LMOA_{\mathcal D,0}\).  Then \(\alpha>0\), and the construction above
gives $
e_n\xrightarrow{w}0
$
in $c_0$.
Therefore
$
 u_n=\iota(e_n)\xrightarrow{w}0
$
in $X$. On the other hand,
\[
 \inf_n\|T_gu_n\|_X\gtrsim\alpha>0.
\]
Hence \(T_g\) is not completely continuous, which completes the proof.
\end{proof}

\begin{proof}[Proof of Corollary~\ref{cor:polynomial}]
Let \(g\) be a Dirichlet polynomial. Then
\(g\in\VMAD\subset\BMAD\) and \(g'\in \mathcal H^\infty\). Hence, for
\(t\in\mathbb R\) and \(0<h\leq\delta\leq1\),
\[
\begin{aligned}
 [g]_{\LMAD,\delta}
 &\lesssim\|g'\|_{\mathcal H^\infty}
 \delta\log(e/\delta)\rightarrow0,
\end{aligned}
\]
that is  \(g\in LMOA_{\mathcal D,0}\). By Theorem~\ref{thm:essential}, we 
know that \(T_g\) is compact on both \(\BMAD\) and \(\VMAD\).
\end{proof}

\section{Endpoint compactness on Bergman spaces of Dirichlet series}
In this section, we prove Theorem~\ref{thm:bergman-polynomial}. The argument adapts the method used by Fu, Guo and Yan for the McCarthy-Bergman spaces
\cite[Lemma~5.9]{FuGuoYan}. For a general admissible measure \(\mu\), the key step is to prove that each positive horizontal translation is compact on \(\mathcal A_\mu^1\). The same semigroup argument then shows that the primitive operator is compact. This gives the result for Dirichlet polynomial symbols.
  For \(u\geq0\), define
$
P_uf(s)=f(s+u).
$
If \(f(s)=\sum_{n\geq1}a_n n^{-s}\), write
\[
 S_Nf(s)=\sum_{n\leq N}a_n n^{-s}.
\]

\begin{lemma}\label{lem:translation-compact-Amu}
Let \(\mu\) be an admissible measure. Then, for every \(u>0\),
\(P_u\) is compact on \(\mathcal A_\mu^1\).
\end{lemma}

\begin{proof}
Fix \(u>0\).  By \cite[Lemma~5.4]{GomezCabello},
\(P_u\) is a contraction
on \(\mathcal A_\mu^1\).  Moreover, by
\cite[Theorem~5.2]{BondarenkoBrevigSaksmanSeip}, for all sufficiently
large \(N\), we get
\[
 \|S_N\|_{\mathcal H^1\to\mathcal H^1}
 \lesssim
 \frac{\log N}{\log\log N}.
\]
Since \((S_Nf)_\sigma=S_N(f_\sigma)\), for every Dirichlet polynomial \(f\)
we have
\[
\begin{aligned}
 \|S_Nf\|_{\mathcal A_\mu^1}
 =\int_0^\infty\|S_N(f_\sigma)\|_{\mathcal H^1}
   \,\dd\mu(\sigma)\lesssim
 \frac{\log N}{\log\log N}
 \|f\|_{\mathcal A_\mu^1}.
\end{aligned}
\]
Hence, by density, we have
\begin{equation}\label{eq:SN-Amu1}
 \|S_N\|_{\mathcal A_\mu^1\to\mathcal A_\mu^1}
 \lesssim
 \frac{\log N}{\log\log N},
 \quad N\to\infty.
\end{equation}

By \cite[Theorem~9(i)]{BL}, each coefficient functional is
continuous on \(\mathcal A_\mu^1\). Hence, by density, the bounded
extension of \(S_N\) satisfies
\[
 S_Nf(s)=\sum_{n\leq N}a_n(f)n^{-s},
 \quad f\in\mathcal A_\mu^1.
\]
In particular, we have
$
 a_n\bigl((I-S_N)f\bigr)=0,
 1\leq n\leq N.
$
Therefore, \cite[Lemma~5.5]{GomezCabello} gives
\[
 \|P_u(I-S_N)f\|_{\mathcal A_\mu^1}
 \leq
 (N+1)^{-u}\|(I-S_N)f\|_{\mathcal A_\mu^1}.
\]
Together with \eqref{eq:SN-Amu1}, this yields

\[
\begin{aligned}
 \|P_u-P_uS_N\|
 &\leq
 (N+1)^{-u}\bigl(1+\|S_N\|\bigr)\\
 &\lesssim
 (N+1)^{-u}
 \left(1+\frac{\log N}{\log\log N}\right)
 \rightarrow0.
\end{aligned}
\]
Since \(P_uS_N\) has finite rank,
\(P_u\) is compact.
\end{proof}

\begin{lemma}\label{lem:primitive-compact-Amu}
For \(f\in\mathcal A_{\mu,0}^1\), define
\[
 Jf=\int_0^\infty P_tf\,\dd t,
\]
where the integral is taken in \(\mathcal A_\mu^1\).  Then
\(J:\mathcal A_{\mu,0}^1\to\mathcal A_{\mu,0}^1\) is compact.
Moreover, if
\(f(s)=\sum_{n\geq2}a_n n^{-s}\), then
\[
 Jf(s)=\sum_{n\geq2}\frac{a_n}{\log n}n^{-s}.
\]
\end{lemma}

\begin{proof}
We first verify that \(J\) is well defined.
By
\cite[Lemma~5.4]{GomezCabello},
\(\|P_t\|\leq1\) for \(t\geq0\).
For every Dirichlet polynomial \(Q\),
\[
 \|P_tQ-Q\|_{\mathcal A_\mu^1}\rightarrow0,
 \quad t\rightarrow0.
\]
If \(f\in\mathcal A_\mu^1\), then
\[
 \|P_tf-f\|_{\mathcal A_\mu^1}
 \leq
 2\|f-Q\|_{\mathcal A_\mu^1}
 +\|P_tQ-Q\|_{\mathcal A_\mu^1}.
\]
Approximating \(f\) by Dirichlet polynomials shows that
\((P_t)_{t\geq0}\) is a strongly continuous semigroup.
For \(f\in\mathcal A_{\mu,0}^1\),
using Lemma~5.5 in \cite{GomezCabello} with \(N=2\),
gives
\[
 \|P_tf\|_{\mathcal A_\mu^1}
 \leq2^{-t}\|f\|_{\mathcal A_\mu^1}.
\]
Consequently, the defining Bochner integral converges and
\[
 \|J\|
 \leq\int_0^\infty2^{-t}\,\dd t
 =\frac1{\log2}.
\]
Since \(\mathcal A_{\mu,0}^1=\ker a_1\) is closed and
\(a_1(P_tf)=a_1(f)=0\) for \(t\geq0\), the Bochner integral defining
\(Jf\) belongs to \(\mathcal A_{\mu,0}^1\).

We next prove compactness.  For \(u>0\), the restriction of
\(P_u\) to \(\mathcal A_{\mu,0}^1\) is compact
by Lemma~\ref{lem:translation-compact-Amu}; since \(J\) is bounded,
\(JP_u\) is compact.  The semigroup identity gives
\[
 JP_uf=\int_u^\infty P_tf\,\dd t,
\]
and therefore
\[
 (J-JP_u)f=\int_0^uP_tf\,\dd t.
\]
Since \(\|P_t\|\leq1\), we have
$
 \|J-JP_u\|\leq u\rightarrow0,
  u\rightarrow0.
$
Hence \(J\) is compact.

Finally, for every \(n\geq2\), we get
\[
\begin{aligned}
 a_n(Jf)
 =\int_0^\infty a_n(P_tf)\,\dd t=a_n(f)\int_0^\infty n^{-t}\,\dd t=\frac{a_n(f)}{\log n}.
\end{aligned}
\]
Therefore, uniqueness of the Dirichlet coefficients gives
\[
 Jf(s)=\sum_{n\geq2}\frac{a_n(f)}{\log n}n^{-s}.
\] The proof is complete.
\end{proof}

\begin{proof}[Proof of Theorem~\ref{thm:bergman-polynomial}]
For \(k\geq2\), define
$
 M_kf(s)=k^{-s}f(s).
$
As observed in the proof of
\cite[Proposition~7.12]{GomezCabello}, multiplication by a Dirichlet
monomial is contractive on
\(\mathcal A_\mu^1\), and
$
 M_k(\mathcal A_\mu^1)\subset\mathcal A_{\mu,0}^1.
$
For every Dirichlet polynomial \(f\), we have
\[
\begin{aligned}
 T_{k^{-s}}f(s)
  =(\log k)\int_0^\infty k^{-(s+t)}f(s+t)\,\dd t=(\log k)J(M_kf)(s).
\end{aligned}
\]
By Lemma~\ref{lem:primitive-compact-Amu}, \(JM_k\) is compact on
\(\mathcal A_\mu^1\).  Thus,  
 \(T_{k^{-s}}\) is compact on \(\mathcal A_\mu^1\) by density.
Finally,  \(T_g\) is compact as a finite sum of compact operators \(T_{k^{-s}}\).
\end{proof}

We use the weighted Bloch space
\(\operatorname{Bloch}_\mu(\Czero)\) and its little Bloch subspace
\(\operatorname{Bloch}_{\mu,0}(\Czero)\) introduced in
\cite[Definitions~3.10 and~3.25]{GomezCabello}, we refer to that paper for
their definitions and basic properties.  In view of Theorem~\ref{thm:bergman-polynomial}, the proof of
\cite[Theorem~7.13]{GomezCabello} also applies at the endpoint \(p=1\), with
\cite[Proposition~7.12]{GomezCabello} replaced by
Theorem~\ref{thm:bergman-polynomial}.  This yields the following result.

\begin{corollary}\label{cor:bergman-endpoint}
	For $1 \leq p<\infty.$ Let \(\mu\) be an admissible measure satisfying the conditions in
	\cite[Definition~3.10]{GomezCabello} and the \(H_1\)-condition in
	\cite[Definition~3.19]{GomezCabello}, and let
	\(g\in \operatorname{Bloch}_{\mu,0}(\Czero)\cap\cD\).  Then the operator
	$T_g$ is compact on $\mathcal A_\mu^p$. 
\end{corollary}

\end{document}